\documentclass{article}

\usepackage{hyperref}
\usepackage[english]{babel}

\usepackage{graphicx}
\usepackage{amssymb}
\usepackage{amsthm, amsmath}
\usepackage[all]{xy}
\usepackage{pstricks,pstricks-add,pst-node,pst-text,pst-3d}
\usepackage{graphicx}
\usepackage{psfrag}
\usepackage{multirow}

\usepackage[applemac]{inputenc}
\def\Rb{\mathbb{R}}		%reals
\newcommand{\rk}{\mathit{rank }}

\newcommand{\restr}[1]
   {\vrule height1ex width.4pt
depth1.4ex\lower1.4ex\hbox{\scriptsize $\,#1$}}

\newtheorem{theorem}{Theorem}

\newtheorem{lemma}{Lemma}

\theoremstyle{definition}

\newtheorem{rem}{Remark}

\author{Inês Cruz\thanks{Centro de Matem\'atica da Universidade do Porto, Departamento de Matem\'atica, Faculdade de Ci\^encias da Universidade do Porto, 4169-007 Porto, Portugal}\,  and M. Esmeralda Sousa-Dias\thanks{Departamento de Matemática, Center for Mathematical Analysis, Geometry, and Dynamical Systems (CAMGSD-LARSyS),  Instituto Superior Técnico, 1049-001 Lisboa, Portugal.}}

\begin{document}
\title{Discrete dynamical systems from mutation-periodic quivers: examples and reduction}
\maketitle

 \begin{abstract}
Several new mutation-periodic quivers of period higher than 1 are introduced as well as the associated discrete dynamical systems.  The reduction of these systems is developed using either a presymplectic or a Poisson approach. The presymplectic approach leads to a reduced system whose iteration map is symplectic with respect to a log symplectic form. In the Poisson approach we build a Poisson structure invariant under the iteration map, leading to a reduced system whose variables are the Casimirs of such structure.
 \end{abstract}


\section{Introduction} 

We study discrete dynamical systems arising in the context of the theory  of cluster algebras  through the notion of mutation-periodic quivers. These systems are given by the iterates $\varphi^{(n)}$ of a map $\varphi$ which is the composition of mutations (at the nodes of the quiver) and a permutation. The map $\varphi$ is birational and the corresponding dynamical system is nonlinear.
\medskip

Motivated by the so-called Laurent phenomenon for cluster algebras (each cluster variable is a Laurent polynomial in the initial  cluster variables) Fordy and Marsh introduced in \cite{FoMa} the notion of mutation-periodicity of a  quiver   in order to  show that $m$-periodic quivers give rise to discrete dynamical systems defined by $m$ difference equations. The classification of 1-periodic quivers as well as a few examples of 2-periodic quivers with a low number of nodes ($\leq 5$) were also obtained in that reference.

The interplay between the theory of cluster algebras and of Poisson and presymplectic structures (developed by \cite{GeShVa2, GeShVa} and surveyed in  \cite{GeShVa3}) opened new research directions in the study of the maps arising from mutation-periodic  quivers.  In particular, in \cite{ FoHo, FoHo2} presymplectic structures were used to show that  some maps  associated to 1-periodic quivers can be reduced to symplectic maps. This reduction enabled the authors  to study the integrability of these maps. Namely they showed that the reduced Somos-4 and Somos-5 recurrence relations are integrable, being the Somos-5 reduced iteration map an instance of a symmetric QRT map (see \cite{QRT}) whose dynamics have been comprehensively studied by Duistermaat in \cite{Duist}.

As the classification of $m$-periodic quivers is still open for $m>1$, the dynamical systems associated to  these quivers have not received too much attention. However, in \cite{InEs2}  we proved that maps arising from  quivers of arbitrary period may also be reduced to a symplectic map, thus generalizing  the 1-periodic case. Due to the complexity of these nonlinear systems, we see this reduction as an important step towards the study of the respective dynamics.

In this work  we present some new quivers of period higher than 1 and we obtain for each of them a reduced system using either a presymplectic or a Poisson approach.

The paper is organized as follows. We start by fixing notation and reviewing the necessary notions of the theory of cluster algebras. In Section~\ref{sec3} we present new quivers of periods 2 and 3 and describe the dynamical systems associated to them. The last section is devoted to the reduction of the iteration map of such dynamical systems. For this purpose, in the presymplectic approach we use the results and  techniques developed in \cite{InEs2} to reduced the iteration map to a symplectic map, whereas in  the Poisson approach we show that the relevant notion is that of a Poisson structure for which the iteration map is a Poisson map. To the best of our knowledge, the use of Poisson structures to reduce dynamical systems associated to 2-periodic quivers is new.

\section{Preliminary notions}\label{sec1}

 We will work in the context of coefficient free cluster algebras ${\cal A}(B)$ where $B$ is a (finite) skew-symmetric integer matrix.  Such matrix  $B=[b_{ij}]$ can be seen as representing an oriented graph with no loops nor  2-cycles, where a positive entry $b_{ij}$ stands for the number of arrows from the vertex $i$ to the vertex $j$ in such graph.  This oriented graph will be called a   {\it quiver} $Q$ and when necessary  $B_Q$ will denote the matrix  representing $Q$. Graphically, a quiver with $N$ nodes will be represented   by  a regular $N$-sided polygon with vertices labelled by  $1, 2, \ldots, N$  in clockwise direction. An arrow between  two vertices is weighted by a positive integer $w$, meaning that there are $w$ arrows between these vertices.  
 
 The basic operation in the theory of cluster algebras, introduced by Fomin and Zelevinski in  \cite{FoZe}, is called a mutation. The {\it mutation}  $\mu_k$ at node $k$  (or in direction $k$) of a quiver $Q$   transforms $Q$ according to the rules: {\it all arrows which originate or terminate at  node $k$ are reversed} and {\it all the other arrows are transformed according to the rules described in the diagrams of Figure~\ref{fig2}}.
 
\begin{figure}[ht]
\begin{center}
\psfrag{i}{$i$}
\psfrag{j}{$j$}
\psfrag{k}{$k$}
\psfrag{p}{$p$}
\psfrag{q}{$q$}
\psfrag{r}{$r$}
\psfrag{a}{$r-pq$}
\psfrag{r+pq}{$r+pq$}
\psfrag{Q}{$Q$}
\psfrag{muk}{$\mu_k$}

\resizebox{!}{7.0cm}{\includegraphics{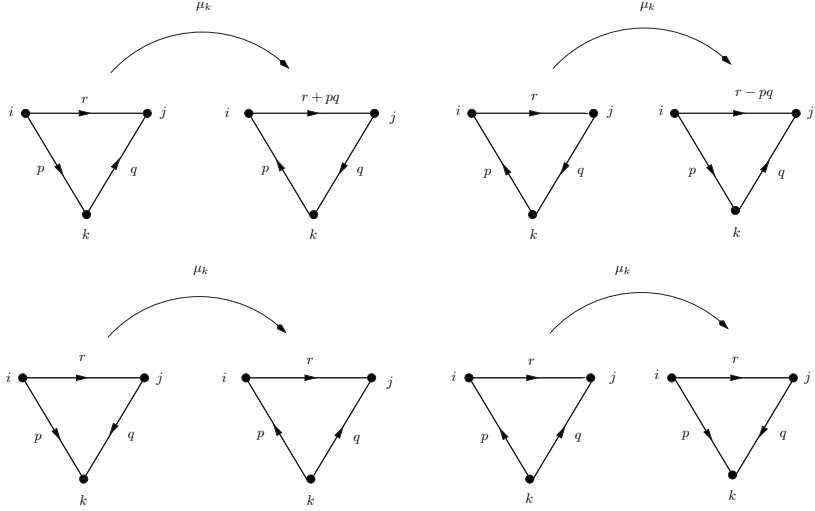}}
\caption{\label{fig2} Diagram exchange arrows for the mutation at node $k$ (in the upper-right situation, if $r-pq<0$ the arrow goes from $j$ to $i$ with weight $pq-r$).}
\end{center}
\end{figure}

In terms of the matrix $B=[b_{ij}]$ representing $Q$, the mutation $\mu_k$ translates into:
\begin{itemize}
\item $\mu_k(B) =\left[b'_{ij}\right]$, with
\begin{equation}\label{mut k2}
b'_{ij}=\begin{cases}
-b_{ij},&\text{$ (k-i)(k-j)=0$}\\
b_{ij}+\frac{1}{2} \left(|b_{ik}|b_{kj}+ b_{ik}|b_{kj}|\right),&\text{otherwise}.
\end{cases}
\end{equation} 
\end{itemize}
\begin{rem}\label{rem1}
It follows from \eqref{mut k2} that, for all  $i,j$ different from $k$,  $b'_{ij}=b_{ij}$ whenever  $b_{ik}$ and $b_{kj}$ have different signs.
\end{rem}

To each node $i$ of a quiver $Q$  one attaches a variable $x_i$ called  {\it cluster variable}. The pair $(B,\mathbf{x})$ is called the {\it initial seed} where  $\mathbf{x}=(x_1,\ldots,x_N)$ is known as the {\it initial cluster}.

The mutation $\mu_k$ acts on the initial seed $(B,\mathbf{x})$ by the action \eqref{mut k2} on $B$ and by the following action on the initial cluster  $\mathbf{x}$:
\begin{itemize}
\item $\mu_k(x_1, \ldots, x_N)=  (x'_1,\ldots, x'_N)$, with
\begin{equation}\label{mut k1}
x'_i=\begin{cases}
x_i& i\neq k\\
\displaystyle{\frac{\prod_{j: b_{kj}>0}x_j^{b_{kj}}+\prod_{j: b_{kj}<0}x_j^{-b_{kj}}}{x_k},}&i=k.
\end{cases}
\end{equation}
Whenever  one of the products in \eqref{mut k1} is taken over an empty set,  its value is assumed to be 1.
\end{itemize}

The formulae  \eqref{mut k2} and \eqref{mut k1} are called (cluster) {\it exchange relations}.
\bigbreak

\noindent{\bf Mutation-periodicity}
\bigbreak

A discrete dynamical system (in the cluster variables) can be associated to a quiver whenever the quiver has a mutation-periodicity property (notion introduced in  \cite{FoMa}). To define a mutation-periodic quiver, let us consider the permutation
$$\sigma=\left(\begin{matrix}
1&2&\cdots&N-1&N\\
2&3&\cdots&N&1
\end{matrix}\right).$$
The action of $\sigma$ on a quiver $Q$ with $N$ nodes  leaves its arrows fixed and moves the nodes in the counterclockwise direction. In terms of the matrix $B_Q$, the action $Q\mapsto\sigma Q$ is equivalent to 
$B_{\sigma Q} = \sigma^{-1} B_Q\, \sigma,$
where, slightly abusing notation, $\sigma$ also denotes the matrix representing the permutation, that is
\begin{equation}
\label{sigma}
\sigma = \begin{bmatrix}
0&1&0&\cdots&0\\
0&0&1&\cdots&0\\
\vdots&\vdots&\ddots&\ddots&\vdots\\
\vdots&\vdots&\vdots&\ddots&1\\
1&0&0&\cdots&0
\end{bmatrix} .
\end{equation}

A quiver $Q$  is said to be {\it $m$-periodic}   if  $m$  is the smallest positive integer such that
$$Q=Q(1)\stackrel{\mu_1}{\longrightarrow}Q(2)\stackrel{\mu_2}{\longrightarrow}Q(3)\cdots \stackrel{\mu_{m-1}}{\longrightarrow}Q(m)\stackrel{\mu_m}{\longrightarrow}Q(m+1)=\sigma^mQ(1).
$$
Equivalently, $Q$ is $m$-periodic if 
\begin{equation}\label{matrixper}
\mu_m\circ\cdots\circ\mu_1(B_Q)=\sigma^{-m} B_Q\, \sigma^m.
\end{equation}

In  Figure~\ref{per24} we  illustrate   the notion of mutation-periodicity  with  a 2-periodic quiver of  4 nodes ($r,s,t,p$ are positive integers with $(r,s)\neq (t,p)$ - in other cases the quiver is 1-periodic). We note  that this quiver is new, since it is more general than the family of  2-periodic quivers  with 4 nodes obtained in \cite{FoMa}.

\begin{figure}[htb]
\begin{center}
\psfrag{mu1}{$\mu_1$}
\psfrag{mu2}{$\mu_2$}
\psfrag{11}{$\rnode{c}{\pscirclebox{3}}\,\,1$}
\psfrag{22}{2\,$\rnode{c}{\pscirclebox{4}}$}
\psfrag{t+rs}{$t+rs$}
\psfrag{t}{$t$}
\psfrag{s}{$s$}
\psfrag{r}{$r$}
\psfrag{p}{$p$}
\psfrag{33}{3\,$\rnode{c}{\pscirclebox{1}}$}
\psfrag{44}{$\rnode{c}{\pscirclebox{2}}\,\, 4$}
\psfrag{1}{$1$}
\psfrag{2}{$2$}
\psfrag{3}{$3$}
\psfrag{4}{$4$}
\psfrag{a}{$1$}
\psfrag{b}{$2$}
\psfrag{c}{$3$}
\psfrag{d}{$4$}
\psfrag{Q}{$Q$}

\resizebox{!}{5cm}{\includegraphics{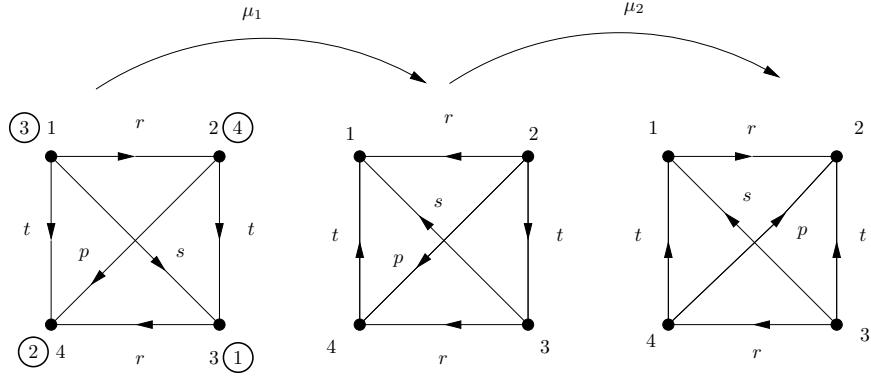}}
\caption{\label{per24} On the left, the quivers $Q$ and  $\sigma^2 Q$ (with circled nodes) are superposed. The middle and right quivers are, respectively, $\mu_1(Q)$ and $\mu_2\circ\mu_1 (Q)$. The quiver $Q$ is 2-periodic since $\sigma^2 Q=\mu_2\circ\mu_1 (Q)$. \label{figquiver}}
\end{center}
\end{figure}

\section{Discrete dynamical systems from  quivers of period higher than 1}\label{sec3}

We recall that  the classification of quivers of period greater than 1 has not yet been achieved although some examples  are known (see \cite{FoMa}). In this section  we present some new quivers of periods 2 and 3 and describe the corresponding discrete dynamical system. 
\bigbreak

\begin{itemize}
\item[A)] {\bf A 2-periodic quiver with 4 nodes}

Let us explain how one constructs the dynamical system associated to an $m$-periodic quiver using  the   2-periodic quiver of Figure~\ref{figquiver}. This quiver is represented by the matrix 
\begin{equation}\label{mat4nos}
B=\begin{bmatrix}
0&r&s&t\\
-r&0&t&p\\
-s&-t&0&r\\
-t&-p&-r&0
\end{bmatrix},
\end{equation}
where $r,s,t,p$ are positive integers with $(r,s)\neq(t,p)$. Using the exchange relations \eqref{mut k2} we have
$$\mu_1(B)=\begin{bmatrix}0&-r&-s&-t\\
r&0&t&p\\
s&-t&0&r\\
t&-p&-r&0
\end{bmatrix}\Longrightarrow \mu_2\circ\mu_1(B)=\begin{bmatrix}
0&r&-s&-t\\
-r&0&-t&-p\\
s&t&0&r\\
t&p&-r&0
\end{bmatrix},
$$
which satisfies \eqref{matrixper} with $m=2$. Taking the initial cluster $\mathbf{x}=(x_1,x_2,x_3,x_4),$
the exchange relations \eqref{mut k1} produce  for the mutation at node 1 followed by the mutation at node 2  the  clusters $\mu_1(\mathbf{x})=(x_5, x_2, x_3, x_4)$ and  $\mu_2 \circ \mu_1(\mathbf{x})=(x_5, x_6, x_3, x_4)$ with
\begin{align}\label{2p4n11}
x_5x_1&= x_2^rx_3^sx_4^t+1\\
\label{2p4n2}
 x_6x_2&=x_3^tx_4^px_5^r+1,
\end{align}
where $x_5=x'_1$ and $x_6=x'_2$. 

We note that   relations \eqref{2p4n11} and \eqref{2p4n2} can be read directly from the first row of $B$ and the second row of $\mu_1(B)$, respectively. As the quiver is 2-periodic the third row of $\mu_2\circ\mu_1(B)$ is just a permutation of the first row of $B$ implying that the cluster $\mu_3\circ\mu_2\circ\mu_1(\mathbf{x})=(x_5,x_6,x_7,x_4)$ which is given by
$$x_7x_3=x_4^rx_5^sx_6^t+1, 
$$
is just  a  2-shift of relation \eqref{2p4n11}. Likewise, mutating next at node 4 one obtains a relation which is a 2-shift of \eqref{2p4n2}. Therefore successive mutations of the 2-periodic quiver at consecutive nodes give rise to the following system of difference equations
\begin{equation}\label{4nos}
 \left\{\begin{array}{ll}
 x_{2n+3}x_{2n-1}&=x_{2n}^rx_{2n+1}^sx_{2n+2}^t+1\\
 &\\ \nonumber
 x_{2n+4}x_{2n}&=x_{2n+1}^t x_{2n+2}^px_{2n+3}^r+1
 \end{array}\qquad n=1,2,\ldots\right.
 \end{equation}
with  initial cluster $(x_1,x_2,x_3, x_4)$. The iteration map corresponding to this dynamical system is  
%$\varphi(x_1,y_1,x_2,y_2) = (x_2,y_2,x_3,y_3)$ with
\begin{equation}\label{map4nos2per}
\varphi(x_1,x_2,x_3, x_4) =  (x_3,x_4,\underbrace{\frac{x_2^rx_3^sx_4^t+1}{x_1}}_{x_5},\underbrace{\frac{x_3^tx_4^px_5^r+1}{x_2}}_{x_6})
%x_3=\frac{y_1^rx_2^sy_2^t+1}{x_1}\quad  \mbox{and} \quad y_3=\frac{x_2^ty_2^px_3^r+1}{y_1}.
\end{equation}
 which, in terms of the initial cluster $\mathbf{x}$, is just  
 \begin{equation}\label{map4n}
 \varphi(\mathbf{x})=\sigma^2\circ\mu_2\circ\mu_1 (\mathbf{x}).
 \end{equation}
 
 \item[B)] {\bf A 2-periodic quiver with 6 nodes}
 
We now consider  the 2-periodic quiver with 6 nodes presented in \cite{InEs2}.  
 
 It is easy to check that the  matrix 
 \begin{equation}\label{matrix62per}
B=\begin{bmatrix}
0&-r&s&-p&s&-t\\
r&0& -t-rs&s&-p-r s&s\\ 
-s&  t+rs&0& -r-s(t-p)&s&-p\\
p&-s& r+s(t-p)&0& -t-r s&s\\
-s&p+r s&-s& t+r s&0&-r\\
t&-s&p&-s&r&0
\end{bmatrix}
\end{equation}
 represents a 2-periodic quiver when $r,s,t,p$ are positive integers and $r\neq t$.

\begin{rem}
 In the case $r=t$ the quiver is 1-periodic and taking the initial cluster to be $ \mathbf{x}=(x_1,x_2,\ldots ,x_6)$ one obtains a recurrence relation of order  6 which can be read directly from the first row of $B$, notably
 \begin{equation}\label{rec6nodes1p}
x_{n+6} x_n= x_{n+1}^rx_{n+3}^px_{n+5}^r+x_{n+2}^sx_{n+4}^s, \quad n=1,2,\ldots \nonumber
\end{equation}
In this 1-periodic situation the iteration map is simply $\varphi(\mathbf{x})=\sigma\circ\mu_1 (\mathbf{x})$. 
\end{rem}

When the quiver is 2-periodic (i.e.,  $r\neq t$) the corresponding system can be read from the first row of $B$ and the second row of $ \mu_1(B)$ which are respectively $(0,-r,s,-p,s,-t)$ and $(-r,0,-t,s,-p,s)$.  Therefore, the discrete dynamical system is 
\begin{equation}\label{rec6nodes}
\left\{\begin{array}{ll}
x_{2n+5}x_{2n-1}&=x_{2n}^rx_{2n+2}^px_{2n+4}^t+x_{2n+1}^sx_{2n+3}^s\\
&\\ \nonumber
x_{2n+6}x_{2n}&=x_{2n+1}^t x_{2n+3}^px_{2n+5}^r+x_{2n+2}^sx_{2n+4}^s
\end{array}n=1,2,\ldots
\right.
\end{equation}
 The respective iteration map is again given by \eqref{map4n}, or explicitly by 
\begin{equation}\label{map6nos2per}
\varphi(x_1,x_2,x_3,x_4,x_5,x_6) =  (x_3,x_4,x_5,x_6,\underbrace{\frac{x_2^rx_4^px_6^t+x_3^sx_5^s}{x_1}}_{x_7},\underbrace{\frac{x_3^tx_5^px_7^r+x_4^sx_6^s}{x_2}}_{x_8} ). \nonumber
\end{equation}

\item[C)] {\bf A 2-periodic quiver with 5 nodes}

This quiver was considered in \cite{FoMa} where the corresponding system was described as a recurrence relation of order 2 in the plane subject to a boundary condition. Our interpretation of the system does not require any boundary condition. 

Let  $\mathbf{x}=(x_1,x_2, x_3,x_4,x_5)$ be  the initial cluster,  $r$ and $s$   distinct  positive integers, and $B$ the  matrix
\begin{equation}\label{B5nodes2per}
B=\begin{bmatrix}
0&-r&1&1&-s\\
r&0&-r-s&1-r&1\\
-1&r+s&0&-r-s&1\\
-1&-r-1&r+s&0&-r\\
s&-1&-1&r&0
\end{bmatrix}.
\end{equation}
 Since $B$ represents a 2-periodic quiver and the second row of $\mu_1(B)$ is $(-r,0,-s,1,1)$, the associated discrete dynamical system is
 \begin{equation}\label{5nosper2}
 \left\{\begin{array}{ll}
 x_{2n+4}x_{2n-1}&=x_{2n}^rx_{2n+3}^s+x_{2n+1}x_{2n+2}\\  \nonumber
 x_{2n+5}x_{2n}&=x_{2n+1}^sx_{2n+4}^r+x_{2n+2}x_{2n+3}
 \end{array}\qquad n=1,2,\ldots\right.
 \end{equation}
This is a system of difference equations whose iteration map is still given by \eqref{map4n}. Explicitly, 
\begin{equation}\label{map5nos2per}
\varphi(x_1,x_2,x_3,x_4,x_5) =  (x_3,x_4,x_5,\underbrace{\frac{x_2^rx_5^s+x_3x_4}{x_1}}_{x_6},\underbrace{\frac{x_3^sx_6^r+x_4x_5}{x_2}}_{x_7} ). 
\end{equation}
 
 \bigbreak
 
 \item[D)] {\bf A 3-periodic quiver with 3 nodes}
 
 Our last example is a 3-periodic quiver with  3 nodes. Considering $r,s$ and $t$ to be positive integers, Remark~\ref{rem1} and the fact that $\sigma^3=Id$  enable us to construct  such a quiver. Indeed the matrix
 \begin{equation}\label{3per}
 B=\begin{bmatrix}
 0&r&s\\
 -r&0&t\\
 -s&-t&0
 \end{bmatrix}
 \end{equation}
 represents  a 3-periodic quiver since $B=\mu_3\circ\mu_2\circ\mu_1 (B)$. As the second row of $\mu_1(B)$ is $(r,0,t)$ and the third row $\mu_2\circ\mu_1(B)$ is $(s,t,0)$, the associated system is given by
 \begin{equation}\label{dyn3per}
 \left\{\begin{array}{ll}
 x_{3n+1}x_{3n-2}&=x_{3n-1}^rx_{3n}^s+1\\
x_{3n+2}x_{3n-1}&=x_{3n}^t\,x_{3n+1}^r+1\\
x_{3n+3}x_{3n}&=x_{3n+1}^sx_{3n+2}^t+1
 \end{array}\right.\qquad n=1,2,\ldots
 \end{equation}
 The respective iteration map is $\varphi(\mathbf{x})=\mu_3\circ\mu_2\circ\mu_1(\mathbf{x}).$
 More precisely, 
 \begin{equation}\label{map3nos3per}
 \nonumber
\varphi(x_1,x_2,x_3) =  (\underbrace{\frac{x_2^rx_3^s+1}{x_1}}_{x_4},\underbrace{\frac{x_3^tx_4^r+1}{x_2}}_{x_5},\underbrace{\frac{x_4^sx_5^t+1}{x_3}}_{x_6} ). 
\end{equation}

\begin{rem} 
\label{rem3}
The system \eqref{dyn3per} associated to the quiver represented by \eqref{3per} was obtained  under  the standard framework of the theory of cluster algebras. However we can also associate to \eqref{3per} another system for which the reduction procedure, to be developed in the next section, works exactly in the same way as for \eqref{dyn3per}. Notably, the  system given by the difference equations
 $$\label{dyn3per2}
 \left\{\begin{array}{ll}
 x_{3n+1}x_{3n-2}&=x_{3n-1}^rx_{3n}^s\\
x_{3n+2}x_{3n-1}&=x_{3n}^t\,x_{3n+1}^r\\
x_{3n+3}x_{3n}&=x_{3n+1}^sx_{3n+2}^t
 \end{array}\right.\qquad n=1,2,\ldots
$$
whose iteration map is now
\begin{align} \nonumber
\varphi(x_1,x_2,x_3)&=\left (\frac{x_2^rx_3^s}{x_1},\frac{x_3^tx_4^r}{x_2},\frac{x_4^sx_5^t}{x_3}\right ) \\ \label{per3not}
&=\left(\frac{x_2^rx_3^s}{x_1}, \frac{x_2^{r^2-1}x_3^{t+ rs}}{x_1^r},\frac{x_2^{t r^2-t+r s}x_3^{t^2+trs+s^2-1}}{x_1^{s+t r}}\right). 
\end{align}
 \end{rem}
\end{itemize}

\section{Reduction of the discrete dynamical systems associated to periodic quivers}\label{sec4}

Poisson and presymplectic structures compatible with a cluster algebra were introduced in \cite{GeShVa2} and since then applied to a wide range of subjects  such  as Teichm\"uller spaces,  commutative and non-commutative algebraic geometry and discrete $Y$-systems. We recommend the survey \cite{GeShVa3} and references therein for an account of  several  applications of the theory of cluster algebras.

In the context of discrete dynamical systems arising from periodic quivers,  presymplectic structures  have been used to reduce the original system to a system defined by  a symplectic iteration map. Notably,  using presymplectic structures  \cite{FoHo} and \cite{FoHo2} performed  reduction on systems associated to 1-periodic quivers  and applied it to study the integrability of these systems.  Still in the realm of 1-periodic quivers, in \cite{InEs} we have used both Poisson structures and presymplectic structures to reduce this type of difference equations. 

Concerning quivers of period higher than 1, we showed in \cite{InEs2} that the presymplectic approach  can still be used to reduce systems associated to higher periodic quivers to systems whose iteration map is symplectic. 

 In the present work we will use both  presymplectic  and Poisson structures  to carry out reduction for the systems presented in the previous section. We remark that,  unlike the presymplectic approach  (which can always be used) the Poisson approach may fail to apply. 

\subsection{Reduction via presymplectic forms}

We now describe how to reduce a discrete dynamical system associated to a periodic quiver using a presympletic form. This presymplectic approach is based on our results in \cite{InEs2}. 

Let $Q$ be an $m$-periodic  quiver with $N$ nodes, $B=[b_{ij}]$ the integer matrix representing $Q$ and $(x_1,\ldots, x_N)\in\Rb_+^N$ the initial cluster. The matrix $B$ can be used to define the  presymplectic form
\begin{equation}
\label{ssf}
\omega = \sum_{1\leq i<j\leq N} b_{ij} \frac{dx_i}{x_i}\wedge \frac{dx_j}{x_j}.
\end{equation}
This form will be called  the {\it log presymplectic form associated to $Q$}, since in the coordinates $v_i=\log x_i$ it takes the form
\begin{equation}
\label{ssf2}
\nonumber
\omega = \sum_{1\leq i<j\leq N} b_{ij} dv_i\wedge dv_j.
\end{equation}
 An $m$-periodic quiver gives rise to a discrete dynamical system defined by the iterates $\varphi^{(n)}$ of the map 
 \begin{equation}
 \label{fi}
 \varphi= \sigma^m\circ\mu_m\circ\cdots\circ\mu_1 
 \end{equation} 
 where $\sigma$ is the permutation \eqref{sigma} and $\mu_k$ is the mutation at node $k$ given by \eqref{mut k2} and \eqref{mut k1}.

Theorem 1 of \cite{InEs2} states that  in the $m$-periodic case $\omega$ is invariant under $\varphi$, that is 
\begin{equation}\label{omegainv}
\nonumber
\varphi^*\omega=\omega,
\end{equation}
where $\varphi^*$ stands for  the pullback by $\varphi$. Using this invariance property  and Darboux's theorem for presymplectic forms we proved  that when $\operatorname{rank} B = 2k<N$ the map $\varphi$ can be reduced to a symplectic map $\hat{\varphi}$. Let us  give the precise statement of this result.
\begin{theorem}[\cite{InEs2}]
\label{reducth}
Let $Q$ be an $m$-periodic quiver with $N$ nodes, $B$ the skew-symmetric matrix representing it and $\varphi$ the associated  iteration map in \eqref{fi}. If $\rk\, B=2k<N$ then there exist a submersion $\pi:  \Rb_+ ^N \longrightarrow   \Rb_+^{2k}$ and a map  $\hat{\varphi}: \Rb_+ ^{2k} \longrightarrow  \Rb_+ ^{2k}$ such that 
the following diagram is commutative
$$\xymatrix{ \Rb_+ ^N  \ar @{->} [r]^\varphi\ar @{->} [d]_{\pi} & \Rb_+ ^N \ar @{->}[d]^{\pi}\\
 \Rb_+ ^{2k}  \ar @{->} [r]_{\hat{\varphi}} &  \Rb_+ ^{2k}}$$
Moreover,  $\hat{\varphi}$ is symplectic with respect to  the \emph{canonical log symplectic form}
\begin{equation}
\label{omega0}
\omega_0=\frac{dy_1}{y_1}\wedge \frac{dy_2}{y_2} + \cdots + \frac{dy_{2k-1}}{y_{2k-1}}\wedge \frac{dy_{2k}}{y_{2k}},
\end{equation}
that is $\hat{\varphi}^* \omega_0=\omega_0$.
\end{theorem}

The proof of the above theorem is non constructive since it is based on Darboux's theorem for closed 2-forms of constant rank (see for instance \cite{AbMa} and \cite{SS}).  However, by considering the coordinates $v_i=\log x_i$ and an adequate basis for $\operatorname{range}\, (B)$ given  by the constructive proof  of Theorem 2.3 (E. Cartan) in \cite{LiMa}, the functions $y_1, y_2, \ldots ,y_{2k}$ and the reduced iteration map $\hat{\varphi}$ can be explicitly computed. 
\bigbreak

We apply Theorem~\ref{reducth} to  reduce the iteration map corresponding to some  of the quivers presented in last section.  We note that in the case of quivers with an odd number of nodes it is always possible to  perform reduction since the rank of the matrix representing the quiver is never maximal.  

\begin{itemize}
\item[i)] {\bf Reduction of the  3-periodic quiver in D)}

Let us consider the quiver represented by the matrix $B$ in \eqref{3per} and the initial cluster $(x_1,x_2, x_3)$. The log presymplectic form \eqref{ssf} is
\begin{align*}
\omega&= r \frac{dx_1}{x_1}\wedge \frac{dx_2}{x_2}+ s \frac{dx_1}{x_1}\wedge \frac{dx_3}{x_3}+ t \frac{dx_2}{x_2}\wedge \frac{dx_3}{x_3}\\
&=  r \, d v_1\wedge dv_2+ s\,  d v_1\wedge d v_3+ t \, d v_2\wedge d v_3,
\end{align*}
where $v_i=\log x_i, i=1,2,3$.
 This  form can be written as
$$\omega = d (v_2+s/r v_3)\wedge d (-rv_1+ t v_3)= d\log\left(x_2x_3^{s/r}\right)\wedge d\log\left(\frac{x_3^{t}}{x_1^r}\right).
$$
We note that  the above expression was obtained by considering a basis of $\operatorname{range}\, (B)$ formed by  the first two rows of $B$. More precisely,  by the first row divided by $r$ and the second row, that is the vectors $(0,1, s/r)$ and $(-r,0,t)$. This choice of basis corresponds precisely to the choice in the proof of the aforementioned E. Cartan's Theorem. 

Taking
\begin{equation}\label{redinv3p}
\nonumber
 y_1= x_2x_3^{s/r}, \quad y_2=\frac{x_3^{t}}{x_1^r}
\end{equation}
the form $\omega$ reduces to the canonical log symplectic form $\omega_0$ in \eqref{omega0}, 
and the  map $\pi$ in Theorem~\ref{reducth} is 
$$\pi(x_1,x_2, x_3)= \left (x_2x_3^{s/r},\frac{x_3^{t}}{x_1^r}\right ).$$ 
Using the commutativity of the diagram  in  the same theorem, the reduced map $\hat{\varphi}$  is given by
$$\hat{\varphi} (y_1,y_2) = (\hat{\varphi}_1, \hat{\varphi}_2) = (y_1\circ\varphi,y_2\circ\varphi) .$$

For the sake of simplicity of the reduced map's expression,  we now reduce the dynamical system considered in Remark \ref{rem3} instead of reducing the system \eqref{dyn3per}. In fact, the proof of  Theorem 1 of \cite{InEs2} shows that the reduction theorem still applies to the map \eqref{per3not}, that is to the map
$$\varphi(x_1,x_2,x_3)=\left(\frac{x_2^rx_3^s}{x_1}, \frac{x_2^{r^2-1}x_3^{t+ rs}}{x_1^r},\frac{x_2^{t r^2-t+r s}x_3^{t^2+trs+s^2-1}}{x_1^{s+t r}}\right).$$
After some algebraic  manipulations we obtain
\begin{align*}
\hat{\varphi}_1&=y_1\circ\varphi= y_1^{\frac{r^3-r+r s^2+r^2 s t- st}{r}}y_2^{\frac{s^2+t r s+r}{r^2}}\\
&\\
\hat{\varphi}_2&=y_2\circ\varphi=y_1^{trs+r^2t^2-t^2-r^2}y_2^{\frac{t^2 r+ s t-r}{r}}.
\end{align*}
 Note that $\hat{\varphi}$ depends only on $(y_1,y_2)$ as it should. We also remark that although the original map is a rational map, the reduced map is not necessarily rational. 
 
\item[ii)] {\bf Reduction of the  2-periodic quiver in C)}

We now consider the 2-periodic quiver with 5 nodes represented by the matrix \eqref{B5nodes2per}.  The rank of $B$ is equal to 4 unless $(r,s)=(1,1)$, in which case the quiver would be 1-periodic. To keep the expressions simpler, we will consider the case $r=1$ and $s\neq 1$. The general situation is done in an entirely analogous way. In this case the iteration map $\varphi$ is
\begin{equation}\label{ite5nodes2per}
\nonumber
\varphi(x_1,x_2,x_3,x_4,x_5) = (x_3,x_4, x_5, \underbrace{\frac{x_2x_5^s+x_3x_4}{x_1}}_{x_6}, \underbrace{\frac{x_3^s x_6+x_4x_5}{x_2}}_{x_7}).
\end{equation}
Considering the first row  of $B$ divided by $-1$ and its second row, respectively the vectors $(0,1,-1,-1,s)$ and $(1,0,-(1+s),0,1)$, we construct the form
$$\tilde{\omega}= d(v_2-v_3-v_4+s v_5)\wedge d(v_1-(1+s)v_3+v_5),
$$
where $v_i=\log x_i$. The standard log presymplectic form $\omega$ has rank 4 and $\omega-\tilde{\omega}$  is a  form of rank 2. More precisely
$$\omega-\tilde{\omega} = -(s^2+s-2) dv_3\wedge dv_5,$$
and therefore
\begin{equation}\label{omega5nos}
\omega = d\left(v_2-v_3-v_4+sv_5\right)\wedge d\left(v_1-(1+s)v_3+v_5\right)-(s^2+s-2) dv_3\wedge dv_5.
\end{equation}
So, taking 
\begin{equation}\label{invariantescinco}
\nonumber
y_1=\frac{x_2x_5^s}{x_3 x_4},\quad y_2=\frac{x_1x_5}{x_3^{1+s}},\quad
y_3=\frac{1}{x_3^{s^2+s-2}},\quad y_4= { x_5},
\end{equation}
it is  easy to see that  \eqref{omega5nos}  reduces to the canonical log symplectic form $\omega_0$. Computing the reduced symplectic map $\hat{\varphi}=(\hat{\varphi}_1, \hat{\varphi}_2, \hat{\varphi}_3, \hat{\varphi}_4)$, with $\hat{\varphi}_i=y_i\circ\varphi$, we obtain
\begin{align*}
\hat{\varphi}_1&%=\frac{u_4u_7^t}{u_5u_6}
=\frac{(1+y_1+y_2)^s}{y_1^sy_2^{s-1} (1+y_1)} y_4^{s^2+s-2}\\
\hat{\varphi}_2&%=\frac{u_3}{u_5^{1+t}} u_7
=\frac{1+y_1+y_2}{y_1y_2}\\
\hat{\varphi}_3&%=\frac{1}{u_5^{t^2+t-2}}
=\frac{1}{y_4^{s^2+s-2}}\\
\hat{\varphi}_4&%=u_7
=\frac{ 1+y_1+y_2}{y_1y_2}y_3^{\frac{1}{s^2+s-2}} y_4^{1+s}.
\end{align*}
\end{itemize}

\subsection{Reduction via Poisson Structures}

Poisson structures  which are compatible with a  cluster algebra were introduced in \cite{GeShVa2} and have been applied,  for instance,  to Grassmannians,  to  directed networks  and even to the theory of integrable systems such as Toda flows in $GL_n$ (see for instance \cite{GeShVa3}).
 A Poisson structure is  said to be compatible with a cluster algebra if in any set of cluster variables, the Poisson bracket is given by the simplest possible kind of homogeneous quadratic bracket. 

In the context of reduction of discrete dynamical systems arising from periodic quivers, the relevant notion is not the compatibility of the Poisson structure with the cluster algebra  but that of  a Poisson structure for which the {\it iteration map $\varphi$ is a Poisson map} as we will show in Theorem~\ref{prop1P} below. Although the existence of such a Poisson structure is not guaranteed for arbitrary periodic quivers, we show that it exists for the 2-periodic quivers with an even number of nodes presented in Section~\ref{sec3}.

We consider Poisson brackets of the form
\begin{equation}
\label{PS log}
\{ x_i, x_j\} = c_{ij} x_i x_j, \quad i,j\in \{1,2,\ldots , N\}
\end{equation}
where $C=[c_{ij}]$ is a skew-symmetric matrix and $\mathbf{x}=(x_1,\ldots, x_N)$ is the initial cluster. Recall that a  map $\varphi: M \rightarrow M$ is said to be a Poisson map with respect to a Poisson bracket $\{\, ,\, \}$ (or the bracket $\{\, ,\, \}$ is said to be invariant under $\varphi$) if 
$$\{ f\circ \varphi, g\circ \varphi \} = \{ f, g\} \circ \varphi, \quad \forall f,g \in C^\infty(M).$$
Given a Poisson structure for which the iteration map $\varphi$ is a Poisson map, the 
functions $y_i$ (known as {\it Casimir functions}) such that
$$\{ y_i, f\} = 0, \quad \forall f\in C^\infty(M)$$
 will provide the "reduced" variables.
The next theorem concerning general Poisson structures and the lemma following it fully justify the reduction of our iteration maps via Poisson structures. 

\begin{theorem}\label{prop1P}
Let $\varphi: M \rightarrow M$ be a differentiable map with differentiable inverse, and  $\{ \, , \}$  a Poisson structure on $M$ of non maximal constant rank invariant under $\varphi$.
If  $\{ y_1, \ldots , y_{k}\}$ is a maximal independent set of Casimir functions for $\{ \, , \}$,  then there are functions $\hat{\varphi}_1, \ldots , \hat{\varphi}_{k}$ such that
$$y_i\circ\varphi = \hat{\varphi}_i(y_1,\ldots , y_k), \quad i=1,\ldots ,k.$$
\end{theorem}

\begin{proof}
As $\varphi$ is a Poisson map each function $y_i\circ\varphi$ is again a Casimir:
$$\{ y_i \circ\varphi, f\} = \{ y_i, f\circ\varphi^{-1}\}\circ\varphi  =0.$$
As the set of Casimir functions is the center of the Lie algebra $(C^\infty(M), \{ \, , \})$, the hypothesis on the set $\{ y_1, \ldots , y_{k}\}$ imply that each Casimir is a function of $y_1,\ldots, y_{k}$. 
\end{proof}

In the particular case where $\{ \, ,\}$ is the homogeneous quadratic Poisson structure  \eqref{PS log}, 
the bracket  $\{ \, ,\}$ has constant rank and so the theorem can be applied whenever the rank is not maximal. Moreover the next lemma provides  a  very simple way of computing an independent set of Casimirs. 
\begin{lemma}\label{lema1}
Let $\mathbf{k}=(k_1, \ldots, k_N)\in\mathbb{Z}^N$, $\mathbf{x}=(x_1,\ldots,x_N)$ and $\mathbf{x}^\mathbf{k}=x_1^{k_1}x_2^{k_2}\cdots x_N^{k_N}$. Then $\mathbf{x}^\mathbf{k}$ is a Casimir for the Poisson bracket 
$$\{x_i,x_j\}=c_{ij} x_ix_j$$
if and only if $\mathbf{k}\in \ker C$, where $C=[c_{ij}]$.
\end{lemma}

\begin{proof} For any $i\in \{1, \ldots, N\}$ we have
\begin{align*}
\{x_i, \mathbf{x}^\mathbf{k}\}&=\sum_{j=1}^N c_{ij} k_j\mathbf{x}^\mathbf{k} x_i= \mathbf{x}^\mathbf{k} x_i \sum_{j=1}^N  c_{ij}k_j = \mathbf{x}^\mathbf{k} x_i \left( C\mathbf{k}\right )_i.
\end{align*}
Thus $\mathbf{x}^\mathbf{k}$ is a Casimir if and only if  $C\mathbf{k}=\mathbf{0}$.
\end{proof}

We now reduce the iteration map $\varphi$ of the 2-periodic quivers with an even number of nodes presented in Section~\ref{sec3}, by providing a Poisson structure invariant under $\varphi$. To the best of our knowledge this structure is new in both examples.

\begin{itemize}
\item[iii)] {\bf Reduction of  a 2-periodic quiver in A)}

For the 2-periodic quiver represented in Figure~\ref{per24} and given by  the matrix $B$ in \eqref{mat4nos} there exists a Poisson bracket of the form \eqref{PS log} for which the iteration map $\varphi$ is a Poisson map. In fact, let us consider the Poisson bracket \eqref{PS log} 
with coefficient matrix $C=[c_{ij}]$ given by
 \begin{equation}\label{poisson2p4nos}
 \nonumber 
C=\begin{bmatrix}
0&r&-p&t\\
-r&0&t&-s\\
p&-t&0&r\\
-t&s&-r&0
\end{bmatrix}.
\end{equation}
One can check that the iteration map $\varphi$ in \eqref{map4nos2per} satisfies the relation $\{x_i\circ\varphi\, ,\, x_j\circ\varphi\}=\{x_i,x_j\}\circ\varphi $ and so $\varphi$ is a Poisson map for this  structure.

Both matrices $B$ and $C$ have the same determinant
\begin{equation}\label{det4nosnova}
\nonumber
\det C=\det B= (r^2+t^2-p s)^2,
\end{equation}
therefore $\varphi$  can be reduced to a map in two variables when $\det B=0$. We now choose $(r,s,t,p)=(1,5,3,2)$  which gives $\det B=0$. The respective iteration map is
$$
\varphi(x_1,x_2,x_3,x_4)=(x_3,x_4,\underbrace{\frac{x_2x_3^5x_4^3+1}{x_1}}_{x_5},\underbrace{\frac{x_3^3x_4^2x_5+1}{x_2}}_{x_6}).
$$
A basis  for $\ker C$ is $\{(-5,-3,0,1), (3,2,1,0)\}$ and so by Lemma~\ref{lema1}  the following Casimirs can be taken as  reduced variables:
$$y_1 = \frac{x_4}{x_1^5x_2^3}, \quad y_2 = x_1^3x_2^2x_3.$$
Doing the usual computations we obtain the reduced iteration map
\begin{align*}
\hat{\varphi}_1&=y_1\circ \varphi 
=\frac{1+y_1^2y_2^3(1+y_1^3y_2^5)}{y_1^3y_2^5},\\
&\\
\hat{\varphi}_2&=y_2\circ \varphi 
= y_1^2y_2^3(1+y_1^3y_2^5).
\end{align*}
We note that in this case the reduced map is rational.

\item[iv)] {\bf Reduction of the 2-periodic quiver B)}

Let us consider the 2-periodic quiver represented by the matrix $B$ in \eqref{matrix62per} with $r\neq t$. The $\operatorname{rank}$ of $B$ is not maximal when $p=r+t$.  In this case $\varphi$ is given by 
\begin{equation}
\nonumber
\varphi(x_1,x_2,\ldots,x_6) =  (x_3,x_4,x_5,x_6,\underbrace{\frac{x_2^rx_4^{r+t}x_6^t+x_3^sx_5^s}{x_1}}_{x_7},\underbrace{\frac{x_3^tx_5^{r+t}x_7^r+x_4^sx_6^s}{x_2}}_{x_8} ),
\end{equation}
and there exists a Poisson bracket of the form \eqref{PS log}  for which $\varphi$ is a Poisson map, namely the Poisson tensor with coefficient matrix
$$
C=\begin{bmatrix}
0&1&0&-1&0&1\\
-1&0&1&0&-1&0\\
0&-1&0&1&0&-1\\
1&0&-1&0&1&0\\
0&1&0&-1&0&1\\
-1&0&1&0&-1&0
\end{bmatrix}.
$$
The matrix $C$ has rank 2 and a basis for its kernel is 
$$\{ (-1,0,0,0,1,0),(0,-1,0,0,0,1), (0,1,0,1,0,0), (0,0,1,0,1,0)\}.$$
Again by Lemma~\ref{lema1},  we can consider the following reduced variables
$$y_1=\frac{x_5}{x_1}, \qquad y_2=\frac{x_6}{x_2}, \qquad y_3=x_2x_4, \qquad y_4=x_3x_5.
$$
So, when $p=r+t$ the computation of the reduced iteration map  gives
\begin{align*}
\hat{\varphi}_1&= y_1\circ \varphi =  \frac{y_1\left (y_2^ty_3^{r+t} + y_4^s\right )}{y_4}\\
\hat{\varphi}_2&= y_2\circ \varphi= \frac{y_1^ry_4^t\left (y_2^ty_3^{r+t} + y_4^s\right )^r + y_2^sy_3^s}{y_3}\\
\hat{\varphi}_3&= y_3\circ \varphi=y_2y_3\\
\hat{\varphi}_4&= y_4\circ \varphi = y_1\left (y_2^ty_3^{r+t} + y_4^s\right )
\end{align*}
 and again the reduced map is  rational.
\end{itemize}

\subsection{Conclusions}

The presymplectic approach to reduction can be used whenever $\operatorname{rank} B$ is not maximal, which is not necessarily the case with the Poisson approach. For instance, we can show that the only Poisson structure of the form \eqref{PS log} which is invariant under the iteration map $\varphi$ in \eqref{map5nos2per} is the zero structure. 

Still, when both approaches are available, the Poisson reduction produces simpler expressions for the reduced iteration map. 

\small{}

\end{document}